\documentclass[11pt]{amsart}
\usepackage{setspace}
\usepackage{amsmath}
\usepackage{amsthm}
\usepackage{amssymb}
\usepackage{paralist}
\usepackage{amsfonts}
\usepackage{mathrsfs}
\usepackage{yfonts}

\newtheorem{theorem}{Theorem}[section]
\newtheorem{proposition}[theorem]{Proposition}
\newtheorem{corollary}[theorem]{Corollary}
\newtheorem{lemma}[theorem]{Lemma}
\theoremstyle{definition}
\newtheorem{remark}[theorem]{Remark}
\newtheorem{example}[theorem]{Example}

\def\suchthat{\, \mid \,} 

\def\Spec{\text{Spec}}

\def\Max{\text{Max}}

\def\Rad{\text{Rad}}
\def\Aut{\text{Aut}}
\def\Reg{\text{Reg}}
\def\tq{\text{tq}}
\def\qf{\text{qf}}
\def\m{\textfrak{m}}
\def\q{\textfrak{q}}
\def\p{\textfrak{p}}
\def\n{\textfrak{n}}

\begin{document}

\title[Invariant Ring Extensions]{Ring Extensions Invariant Under\\ Group Action}

\author{Amy Schmidt}

\address{Department of Mathematics, George Mason University, Fairfax, Virginia
22030-4444, {\tt E-mail: aschmid9@masonlive.gmu.edu}}

\begin{abstract}
Let $G$ be a subgroup of the automorphism group of a commutative ring with identity $T$. Let $R$ be a subring of $T$ such that $R$ is invariant under the action by $G$. We show $R^G\subset T^G$ is a minimal ring extension whenever $R\subset T$ is a minimal extension under various assumptions. Of the two types of minimal ring extensions, integral and integrally closed, both of these properties are passed from $R\subset T$ to $R^G\subset T^G$. An integrally closed minimal ring extension is a flat epimorphic extension as well as a normal pair. We show each of these properties also pass from $R\subset T$ to $R^G\subseteq T^G$ under certain group action.
\end{abstract}



\maketitle
\begin{center}
\today
\end{center}

\section{Introduction}
\label{intro}

All rings herein are commutative with identity, and all homomorphisms and subrings are unital. For a ring $R$, we denote by $\Reg(R)$ the set of regular elements; $\Spec(R)$ the set of prime ideals; $\Max(R)$ the set of maximal ideals; $\Rad_R(I)$ the radical in $R$ of an ideal $I\subset R$; $\tq(R)$ the total quotient ring; $\qf(R)$ the quotient field, if $R$ is a domain; and $\Aut(R)$ the automorphism group of $R$. As in \cite{Kaplansky}, we refer to the lying-over, going-up, and incomparable properties of ring extensions as LO, GU, and INC, respectively.

Given a subgroup $G$ of $\Aut(R)$, we say $G$ acts on $R$ and denote the fixed ring of this action by $R^G=\{r\in R\suchthat \sigma (r)=r \text{ for all } \sigma\in G\}$. We say a property of $R$ is \textit{($G$-)invariant} if $R^G$ also has the property. Our purpose in this paper is to enhance the popular investigation of which ring-theoretic properties are invariant. As the title of this paper suggests, we determine properties of the ring extension $R\subseteq T$ that are $G$-invariant, meaning the property descends to the fixed subring extension $R^G\subseteq T^G$. 

\textbf{Our riding assumptions in this work are $R$ is a subring of $T$, $G$ acts on $T$ via automorphisms, and $R$ is $G$-invariant, i.e., $\sigma(R)\subseteq R$ for all $\sigma\in G$.} It then follows that $G$ is a subgroup of $\Aut(R)$. 

We denote the orbit of $t\in T$ under $G$ by $\mathcal{O}_t$, i.e., $\mathcal{O}_t=\{\sigma(t)\suchthat \sigma\in G\}$, and we define
\[
n_t:=|\mathcal{O}_t|,\quad\hat t:=\sum_{t_i\in\mathcal{O}_t}t_i\quad\text{and}\quad \tilde t:=\prod_{t_i\in\mathcal{O}_t}t_i.
\]
If $G$ is finite, instead we denote by $\hat t$ the sum $\sum_{\sigma\in G}\sigma(t)$. We say $G$ is \textit{locally finite} if $\mathcal{O}_T$ is finite for all $t\in T$. Given an ideal $I\subset T$ we denote the orbit of $I$ under $G$ by $\mathcal{O}_I:=\{\sigma(I)\suchthat \sigma\in G\}$. By the First Isomorphism Theorem, $T/I\cong T/\sigma(I)$. Clearly, $T/I$ is a field (domain) if and only if $T/\sigma(I)$ is a field (domain). Hence, $I$ is a maximal (prime) ideal if and only if $\sigma(I)$ is a maximal (prime) ideal. We say $G$ is \textit{strongly locally finite} if $G$ is locally finite and $\mathcal{O}_P$ is finite for all $P\in\Spec(T)$.

As in \cite{FO}, we say $R\subset T$ is a \textit{minimal ring extension} if there is no ring $S$ such that $R\subset S \subset T$. Clearly, this is true if and only if $T=R[u]$ for all $u\in T\backslash R$. Since $R\subseteq \bar R \subseteq T$, where $\bar R$ is the integral closure of $R$ in $T$, if $R\subset T$ is minimal, then either $R$ is integrally closed in $T$, or $T$ is integral over $R$ (equivalently, $T$ is module finite over $R$). In the first case we call $R\subset T$ an \textit{integrally closed minimal ring extension}, and in the second case, we call it an \textit{integral minimal ring extension}. By \cite[Th\'eor\`eme~2.2]{FO}, if $R\subset T$ is a minimal ring extension, there exists a unique maximal ideal $M$ of $R$ such that $R_P\cong T_P$ for all $P\in\Spec(T)\backslash\{M\}$. This maximal ideal is commonly referred to as the \textit{crucial maximal ideal} of the extension. In the integral case, $(R:_RT)$ is the crucial maximal ideal, while in the integrally closed case, $(R:_RT)$ is a prime ideal adjacent to the crucial maximal ideal.

In 1970, Ferrand and Olivier contributed to the groundbreaking work of classifying minimal ring extensions by determining the minimal ring extensions of a field \cite{FO}. More recently, Ayache extended this work to integrally closed domains \cite{Ayache}. Shortly thereafter, Dobbs and Shapiro generalized these results further to arbitrary domains in \cite{DS2006} and then later to certain rings with zero-divisors in \cite{DSMRE}. In their second paper, they completely classify the integral minimal ring extensions of an arbitrary ring, as well as the integrally closed minimal ring extensions of a ring with von Neumann regular total quotient ring  \cite{DSMRE}. In \cite{Picavets} (cf.\ \cite{DMP}), Picavet and Picavet-L'Hermitte give another characterization of integral minimal ring extensions. In \cite{CahenDobbsLucas}, Cahen et al.\ characterize integrally closed minimal ring extensions of an arbitrary ring.

In Section~\ref{integral section}, under the assumption $R\subset T$ is an integral minimal ring extension and $G$ is locally finite acting on $T$ (such that $R$ is $G$-invariant), we show $R^G\subset T^G$ is an integral minimal ring extension under mild hypotheses. To do so we use \cite[Theorem~3.3]{Picavets}, given in Theorem~\ref{integral mre} for reference. We present examples to show it is necessary to assume $R^G\neq T^G$. In one example, we use the idealization construction. Given a ring $R$ and an $R$-module $M$, the \textit{idealization} $R(+)M=\{(r,m)\suchthat r\in R,\,m\in M\}$ is ring with multiplication given by $(r,m)(r',m')=(rr',rm'+r'm)$ and componentwise addition. By \cite[Theorem~2.4]{DobbsMRE}, $R(+)M$ is a minimal ring extension of $R$ if and only if $M$ is a simple $R$-module.

In Section~\ref{integrally closed section}, we turn to the integrally closed case. In Theorem~\ref{integrally closed mre invariance}, we show arbitrary integrally closed minimal ring extensions are $G$-invariant assuming $G$ is locally finite. This invariance is established in \cite[Theorem~3.6]{DS2007Houston} under the assumptions that the base ring is a domain in which $|G|\in\mathbb{N}$ is a unit. The authors use the characterization of the minimal overrings of an integrally closed domain (that is not a field) by Ayache \cite[Theorem~2.4]{Ayache}. This result is generalized by Dobbs and Shapiro \cite[Theorem~3.7]{DSMRE}, and then further generalized by Cahen et al.\ \cite[Theorem~3.5]{CahenDobbsLucas}. The latter authors introduce a new classification of integrally closed minimal ring extensions of an arbitrary ring in terms of rank 1 valuation pairs. 

For an extension $R\subset T$ and a prime ideal $P\subset R$, we say $(R,P)$ is a \textit{valuation pair of $T$} if there exists a valuation $v$ on $T$ such that $R=\{t\in T\suchthat v(t)\geq 0\}$ and $P=\{t\in T\suchthat v(t)>0\}$ as in \cite{Manis} (cf.\ \cite{CahenDobbsLucas}). Equivalently, $(R,P)$ is a valuation pair of $T$ if $R=S$ whenever $S$ is an intermediate ring containing a prime ideal lying over $P$. The \textit{rank} of $(R,P)$ is the rank of the valuation group. A useful necessary and sufficient condition for $(R,P)$ to have rank 1 is that $P$ is a critical ideal \cite[Lemma~2.12]{CahenDobbsLucas}. Cahen et al.\ define a \textit{critical ideal (for $R\subset T$)} as an ideal $I\subset R$ such that $I=\Rad_R((R:_Rt))$ for all $t\in T\backslash R$. That is, $\Rad_R((R:_Rt))$ is the same ideal for all $t\in T\backslash R$. While such an ideal may not exist for some extensions, if it does, clearly it is unique.

In Section 4, we show certain ring extensions related to minimal ring extensions are also invariant. It is easy to see the integral and integrally closed properties are invariant. Other related extensions are flat epimorphic extensions and normal pairs. As in \cite{Davis}, for an extension $R\subset T$, we say $(R,T)$ is a \textit{normal pair} if every intermediate ring is integrally closed in $T$. Clearly integrally closed minimal ring extensions are normal pairs; they are also flat epimorphic extensions \cite[Th\'eorm\`e~2.2]{FO}. (Throughout, we mean epimorphic in the category of commutative rings.) In Proposition~\ref{perfect localization}, we show that flat epimorphic extensions are invariant under strongly locally finite group action. Lastly, in Coroallary~\ref{normal pair invariance}, we assert normal pairs are invariant.

\section{Integral Minimal Ring Extensions}
\label{integral section}

We begin with a well-known result that is fundamental in this work and in much of the work by Dobbs and Shapiro \cite{DS2006Houston}, \cite{DS2007Houston}, \cite{DS2007}. These papers on invariant theory are a strong influence on our work.

\begin{lemma}
\label{integrality}
If $G$ is locally finite, then $T$ is integral over $T^G$.
\end{lemma}

Recall we our riding assumptions in this paper are $R\subset T$, $G$ acts on $T$ and $R$ is $G$-invariant. In the following lemma we establish several technical results needed for the main result of this section. Proposition~\ref{fixed quotient} is another tool for the main result and is also of independent interest.

\begin{lemma} 
\label{conductor lemma}
Assume $G$ is locally finite, and assume $M:=(R:_RT)$ is a maximal ideal of $R$. Set $\m:=M\cap R^G =M\cap T^G$. Then:
\begin{enumerate}[(a)]
\item \label{contraction} The conductor $(R^G:_{R^G}T^G)$ is $\m$.
\item \label{orbit} The orbit of $M$ in $R$ is a singleton set, i.e., $\mathcal{O}_{M}=\{M\}$.
\item \label{lying over M} If there exist $N\in\Spec(T)$ containing $M$, then $M=N\cap R$.
\end{enumerate}
\end{lemma}
\begin{proof}
\begin{inparaenum}[(a)]
\item Let $x\in \m$. Then $x\in R^G$, and $xt\in R$, for all $t\in T$. If $t\in T^G$, then $xt\in T^G$, from which it follows that $xt\in T^G\cap R=R^G$. Hence $x\in(R^G:_{R^G}T^G)$. Thus $\m\subseteq (R^G:_{R^G}T^G)$. By Lemma~\ref{integrality}, $R$ is integral over $R^G$. Hence $\m$ is maximal in $R^G$. Thus $\m=(R^G:_{R^G}T^G)$.

\item Let $\sigma\in G$. Then
\[
\sigma(M)R=\sigma(M)\sigma(R)=\sigma(MR)\subseteq\sigma(R)=R.
\]
Hence $\sigma(M)\subseteq M$. This is sufficient to show $\sigma(M)=M$ in $R$, since $\sigma(M)$ is maximal in $R$, by \cite[Lemma~2.1(b)]{DS2006Houston}.

\item Clearly $M=N\cap R$ whenever $N$ is a prime ideal of $T$ containing $M$, since $M\in\Max(R)$.
\end{inparaenum}
\end{proof}

\begin{proposition}
\label{fixed quotient}
Let $M\in\Max(R)$ and $\m:=M\cap R^G$. Assume $G$ is locally finite such that $\text{char}(R^G/\m)\nmid n_r$ for all $r\in R$. If $\mathcal{O}_M=\{M\}$, then the $G$-action extends to $R/M$ via $\sigma(r+M)=\sigma(r)+M$, for $\sigma\in G$. Moreover $R^G/\m\cong (R/M)^G$.
\end{proposition}
\begin{proof}
The given action of $G$ on $R/M$ is well-defined: If $r+M=s+M$, then $\sigma(r)-\sigma(s)\in\sigma(M)= M$. Hence $\sigma(r)+M=\sigma(s)+M$.

As for the moreover, first note $\m\in\Max(R^G)$, by Lemma~\ref{integrality}. Define $\phi:R^G/\m\rightarrow (R/M)^G$ by $r+\m\mapsto r+M$. Clearly, $\phi$ is a ring homomorphism. If $\phi(r+\m)=0+M$, then $r\in M$. It follows that $r\in M\cap R^G=m$, so $r+\m=0+\m$. Hence $\phi$ is injective. 

Now let $r+M\in (R/M)^G$. Then $r+M=\sigma(r)+M$ for all $\sigma\in G$. Summing $\mathcal{O}_r$ we have $n_rr+M=\hat r+M$. Since $R/M$ is a field, we have $r+M=(n_r+M)^{-1}(\hat r+M)$. Similarly, since $n_r+\m\in R^G/\m$, we have $y+\m:=(n_r+\m)^{-1}\in R^G/\m$. It follows that $y+M=(n_r+M)^{-1}$, whence $\phi(y\hat r+\m)=y\hat r+M=(n_R+M)^{-1}(\hat r+M)=r+M$. Thus $\phi$ is surjective. Hence $R^G/m\cong (R/M)^G$.
\end{proof}

The technique of averaging the orbit of an element used above to produce $r+M=(n_r+M)^{-1}(\hat r+M)$ is introduced in \cite{Bergman}. We generalize this method in the following lemma.

\begin{lemma}
\label{fixed sums}
Assume $G$ is locally finite. Let $t\in T^G$. We show that if $t=r_1u_1+r_2u_2+\cdots+r_ku_k$ for some $r_i\in R$ and $u_i\in T^G$, then there exist $m,m_i\in\mathbb{N}$ and $r_i'\in R^G$ such that $0\neq mt=m_1r_1'u_1+m_2r_2'u_2+\cdots+m_kr_k'u_k$ whenever
\begin{enumerate}[(a)]
\item $T$ is a domain and $\text{char}(T)\nmid n_t$ for all $t\in T$, or
\item $|G|$ is finite and a unit in $T$.
\end{enumerate}
\end{lemma}
\begin{proof}
For all $t\in T$, fix a subset $\mathcal{N}_t$ of $G$ such that for each $a\in\mathcal{O}_t$ there exists a unique $\sigma\in\mathcal{N}_t$ with $a=\sigma(t)$ (and so $|\mathcal{N}_t|=|\mathcal{O}_t|=n_t$).

First we show if 
\begin{equation}
\label{equation 1}
0\neq t=q_1u_1\cdots+q_iu_i+r_{i+1}u_{i+1}+\cdots+r_ku_k,
\end{equation}
where $t\in T^G$, $q_i\in R^G$, and $r_j\in R$, then there exists $m\in\mathbb{N}$, $r_{i+1}'\in R^G$, and $s_j\in R$ such that 
\begin{equation}
\label{equation 2}
0\neq mr=m(q_1u_+\cdots+q_iu_i)+r_{i+1}'u_{i+1}+s_{i+2}u_{i+2}+\cdots+s_ku_k.
\end{equation}
Applying each $\sigma\in\mathcal{N}_{r_{i+1}}$ to (\ref{equation 1}) and summing establishes (\ref{equation 2}). In particular,
\[
m=n_{r_{i+1}},\quad r_{i+1}'=\widehat{r}_{i+1},\quad\text{and}\quad s_j=\sum_{\sigma\in\mathcal{N}_{r_{i+1}}}\sigma(r_j)u_j,
\]
for $i+2\leq j\leq k$. Note $n_{r_{i+1}}r\neq 0$ under assumption (a). Since $i=1$ establishes the base case, the assertion of the lemma now follows from induction. Under assumption (b), the same argument holds replacing $\mathcal{N}_{r_{i+1}}$ with $G$ and $n_{r_{i+1}}$ with $|G|$.
\end{proof}

We have established the machinery needed to prove the main result of this section. We use the characterization provided below for reference.

\begin{theorem}\cite[Theorem~3.3]{Picavets} (cf. \cite[Corollary~II.2]{DMP})
\label{integral mre}
Let $R\rightarrow T$ be an injective ring homomorphism, with conductor $(R:_RT)$. Then $R\rightarrow T$ is minimal and finite if and only if $(R:_RT)\in\Max(R)$ and one of the following three conditions holds:
\begin{enumerate}[(a)]
\item\label{inert} \textbf{Inert case:} $(R:_RT)\in\Max(T)$ and $R/(R:_RT)\rightarrow T/(R:_RT)$ is a minimal field extension.
\item\label{decomposed} \textbf{Decomposed case:} There exist $N_1,N_2\in\Max(T)$ such that $(R:_RT)=N_1\cap N_2$ and the natural maps $R/(R:_RT)\rightarrow T/N_1$ and $R/(R:_RT)\rightarrow T/N_2$ are each isomorphisms.
\item\label{ramified} \textbf{Ramified case:} There exists $N\in\Max(T)$ such that $N^2\subseteq(R:_RT) \subset N$, $[T/(R:_RT):R/(R:_RT)]=2$ and the natural map $R/(R:_RT)\rightarrow T/N$ is an isomorphism.
\end{enumerate}
\end{theorem}

We now present our main result on the invariance of integral minimal extensions.

\begin{theorem}
\label{integral mre invariance}
Let $R\subset T$ be an integral minimal extension with crucial maximal ideal $M=(S:_RR)$. Assume $G$ locally finite such that $R^G\neq T^G$ and $\text{char}(R^G/(M\cap T^G))\nmid n_r$, for all $r\in R$.  Then $R^G\subset T^G$ is a minimal extension of the same type as $R\subset T$. Moreover, the crucial maximal ideal of $R^G\subset T^G$ is $(R^G:_{R^G}T^G)$.
\end{theorem}
\begin{proof}
Throughout the argument, set $\m:=(R^G:_{R^G}T^G)$, whence $\m=M\cap R^G=M\cap T^G$, by Lemma~\ref{conductor lemma}(\ref{contraction}). 

\begin{inparaenum}
\item[\textbf{Inert case:}] By Theorem~\ref{integral mre}(\ref{inert}), $M\in\Max(T)$ and $R/M\rightarrow T/M$ is a minimal field extension. By Lemma~\ref{conductor lemma}(\ref{orbit}) and Proposition~\ref{fixed quotient}, we may pass to $R/M\subset T/M$. Replacing $R/M\subset T/M$ with $R\subset T$, we show $R^G\subset T^G$ is a minimal field extension. Clearly this is true if and only if for all $u\in T^G\backslash R^G$, $T^G=R^G[u]$. If $u\in T^G\backslash R^G$, then $u\in T\backslash R$, so $T=R[u]$. Let $t\in T^G$. Then $t=r_ku^k+\cdots+r_1u+r_0$, for some $k\in\mathbb{N}$ and $r_i\in R$. By Lemma~\ref{fixed sums}, there exist $m,m_i\in\mathbb{N}$ and $r_i'\in R^G$ such that $0\neq mt=m_k r_k'u^k+\cdots+m_1 r_1'u+ m_0r_0'$. Since $R^G$ is a field, we have $t=m^{-1}(m_k r_k'u^k+\cdots+m_1 r_1'u+m_0 r_0')\in R^G[u]$. Hence, $R^G\subset T^G$ is a minimal field extension. By Theorem~\ref{integral mre}(\ref{inert}), the original fixed ring extension (before passing to the quotient ring extension) $R^G\subset T^G$ is an inert integral minimal extension with crucial maximal ideal $\m=(R^G:_{R^G}T^G)$.

\item[\textbf{Decomposed case:}] By Theorem~\ref{integral mre}(\ref{decomposed}), there exist $N_1,N_2\in\Max(T)$ such that $M=N_1\cap N_2$ and the natural maps $R/M\rightarrow T/N_1$ and $R/M\rightarrow T/N_2$ are isomorphisms. Set $\n_1:=N_1\cap T^G$ and $\n_2:=N_2\cap T^G$. Since $T$ is integral over $T^G$, $\n_1,\n_2\in\Max(T^G)$. Clearly
\[
\m=M\cap T^G=(N_1\cap N_2)\cap T^G=\n_1\cap \n_2. 
\]

Define $\phi:R^G/\m\rightarrow T^G/\n_1$ via the natural map $r+\m\mapsto r+\n_1$. Suppose $\phi(r+\m)=0+\n_1$ for some $r\in R^G$. Then $r\in \n_1\cap R^G$, but, by Lemma~\ref{conductor lemma}(\ref{lying over M}), $\n_1\cap R^G=\m$. Hence, $r+\m=0+\m$. Thus, $\phi$ is injective.

To show $\phi$ is surjective, we first note the $G$-action extends to $T/N_1$, since it extends to $R/M$ and $R/M\cong T/N_1$. From Lemma~\ref{conductor lemma}(\ref{orbit}) and Proposition~\ref{fixed quotient}, we have $R^G/\m\cong(R/M)^G\cong(T/N_1)^G$. Let $t+\n_1\in T^G/\n_1$ be nonzero. Then $ t+N_1\in (T/N_1)^G$ is nonzero. (Clearly it is fixed, and if $t\in N_1$, then $t\in N_1\cap T^G=\n_1$ -- contradiction.) Since $R^G/m\cong(T/N_1)^G$ (via composition of the natural maps), there exists $r+\m\in R^G/\m$ such that $r+\m \mapsto r+M\mapsto r+N_1=t+N_1$. It follows $(r-t)\in N_1\cap T^G=\n_1$. Hence $\phi(r+\m)=r+\n_1=t+\n_1$. Thus $\phi$ is surjective, so $R^G/\m\cong T^G/\n_1$. The same argument applies to show $R^G/\m\cong T^G/\n_2$. By Theorem~\ref{integral mre}(\ref{decomposed}), $R^G\subset T^G$ is a decomposed integral minimal extension with crucial maximal ideal $\m=(R^G:_{R^G}T^G)$.

\item[\textbf{Ramified case:}] By Theorem~\ref{integral mre}(\ref{ramified}), there exists $N\in\Max(T)$ such that $N^2\subseteq M \subset N$, $[T/M:R/M]=2$ and the natural map $R/M\rightarrow T/N$ is an isomorphism. Set $\n:=N\cap T^G$, and recall $\m=M\cap T^G$. Clearly, $\n\in\Max(T^G)$ and $\m\subsetneq \n$, since $\m\notin\Max(T^G)$ (since $M\notin\Max(T)$, $N\in\Max(T)$, and $T$ is integral over $T^G$). For the other containment, let $x\in \n^2$. Then $x\in N^2$, so $x\in M$. Hence, $x\in M\cap T^G=\m$. Thus, $\n^2\subseteq \m$.

We show that the natural map $\phi: R^G/\m\rightarrow T^G/\n$ given by $r+\m\mapsto r+\n$ is an isomorphism. Suppose $\phi(r+\m)=0+\n$ for some $r\in R^G$. Then $r\in \n$, so $r^2\in \n^2$. Since $\n^2\subseteq \m$ and $\m$ is prime (maximal) in $R^G$, we have $r\in \m$. (Alternatively, $r\in \n\cap R^G=\m$, by Lemma~\ref{conductor lemma}(\ref{lying over M}).) Hence, $r+\m=0+\m$. Thus, $\phi$ is injective.

Next we show $\phi$ is surjective. Let $t+\n\in T^G/\n$. Then $t+N\in (T/N)^G$. Note that, as in the decomposed case, since $R/M\cong T/N$, the $G$-action extends to $T/N$. From this, Lemma~\ref{conductor lemma}(\ref{orbit}), and Proposition~\ref{fixed quotient}, it follows that $R^G/\m \cong(R/M)^G\cong(T/N)^G$ via $r+\m\mapsto r+M\mapsto r+N$. Hence, there exists $r+\m\in R^G/\m$ such that $r+\m\mapsto r+M\mapsto r+N=t+N$, from which it follows that $(r-t)\in N\cap T^G=\n$. Hence, $\phi(r+\m)=t+\n$. Thus, $\phi$ is surjective.

It remains to show $[T^G/\m:R^G/\m]=2$. Note $T^G/\m$ is not a domain, since $\n^2\subseteq \m\subset n$ implies $\m=\n$, if $\m$ is prime. Hence $T^G/\m\neq R^G/\m$, i.e., $[T^G/\m:R^G/\m]\geq2$.

Suppose $[T^G/\m:R^G/\m]>2$, and let $\{e_1+\m, e_2+\m,e_3+\m\}$ be an $R^G/\m$-linearly independent set in $T^G/\m$. Then each $e_i\notin M$; otherwise, $e_i\in M\cap T^G=\m$. Hence each $e_i+M$ is nonzero in $T/M$. Since $[T/M:R/M]=2$, without loss of generality we may assume there exist $t_1+M,t_2+M\in T/M$ such that 
\[
e_3+M=(t_1+M)(e_1+M)+(t_2+M)(e_2+M)=t_1e_1+t_2e_2+M.
\]
As in Lemma~\ref{fixed sums}, using $\sigma\in \mathcal{N}_{t_1}$ and summing $\mathcal{O}_{t_1}$ we have
\[
n_{t_1}e_3+M=\widehat{t}_1e_1+\left(\sum_{\sigma\in\mathcal{N}_{t_1}}\sigma(t_2)\right)e_2+M.
\]
Defining $t_3$ to be the coefficient of $e_2$ above and repeating the above technique with respect to $t_3$ we have
\[
n_{t_3}n_{t_1}e_3+M=n_{t_3}\widehat{t}_1e_1+\widehat{t}_3e_2+M.
\]
It follows that $n_{t_3}n_{t_1}e_3-(n_{t_3}\widehat{t}_1e_1+\widehat{t}_3e_2)\in M\cap T^G=\m$, so
\[
n_{t_3}n_{t_1}e_3+\m=n_{t_3}\widehat{t}_1e_1+\widehat{t}_3e_2+\m.
\]
Equivalently,
\[
(n_{t_3}n_{t_1}+\m)(e_3+\m)=(n_{t_3}\widehat{t}_1+\m)(e_1+\m)+(\widehat{t}_3+\m)(e_2+\m)
\]
is an $R^G/\m$-linear combination of $e_1+\m,e_2+\m, e_3+\m$ in $T^G/\m$ -- contradiction. Hence, there cannot exist in $T^G/\m$ any more than two $R^G/\m$-linearly independent elements. Thus $[T^G/\m:R^G/\m]\leq 2$. Hence $[T^G/\m:R^G/\m]=2$. By Theorem~\ref{integral mre}(\ref{ramified}), $R^G\subset T^G$ is a ramified integral minimal extension with crucial maximal ideal $\m=(R^G:_{R^G}T^G)$.
\end{inparaenum}
\end{proof}

\begin{remark}
It is necessary to assume $R^G\neq T^G$ in Theorem~\ref{integral mre invariance}, as illustrated in the following.
\end{remark}

\begin{example} The fixed rings are equal, even under finite group action, in the following cases:

\begin{inparaenum}
\item[\textbf{Inert case:}] Set $R:=\mathbb{R}$, $T:=\mathbb{C}$, and $G:=\{1,\sigma\}$, where $\sigma$ is the conjugacy map. Then $R^G=R=T^G$.

\item[\textbf{Decomposed case:}] Let $F$ be a field such that $\text{char}(F)\neq 2$, and set $R:=\{(x,x)\suchthat x\in F\}$ and $T:=F\times F$. By \cite[Lemme~1.2(b)]{FO}, $R\subset T$ is a minimal extension. Define $G:=\{1,\sigma\}$, where $\sigma((x,x)=(x,-x)$. Then $R^G=T^G$. 

\item[\textbf{Ramified case:}] Let $F$ and $R$ be as above, and set $T:=F(+)F$. Then by \cite[Lemme~1.2(c)]{FO}, $R\subset T$ is a minimal extension. Define $G$ as above. Then $R^G=T^G$.
\end{inparaenum}
\end{example}

\section{Integrally Closed Minimal Extension}
\label{integrally closed  section}

In this section, we show that the integrally closed minimal property of the extension $R\subset T$ is invariant under locally finite $G$-action. This generalizes Dobbs' and Shapiro's result  that the property is invariant if $R$ is a domain and if $|G|$ is finite and a unit in $R$ \cite[Theorem~3.6]{DS2007Houston}. They use Ayache's characterization of minimal extensions (overrings) of an integrally closed domain \cite[Theorem~2.4]{Ayache}. Ayache's result has since been generalized by Dobbs and Shapiro \cite[Theorem~3.7]{DSMRE} and recently further generalized by Cahen et al.\ \cite[Theorem~3.5]{CahenDobbsLucas}. In the latter, the authors give several necessary and sufficient conditions for an arbitrary ring extension to be integrally closed  and minimal, which we use to establish Theorem~\ref{integrally closed mre invariance}.

Whereas crucial maximal ideals are historically essential to the study of minimal extensions, Cahen et al.\ introduce critical ideals and use them extensively in characterizing integrally closed minimal extensions of an arbitrary ring \cite{CahenDobbsLucas}. As previously mentioned, they define a critical ideal for $R\subset T$ as an ideal $I\subset R$ such that $I=\Rad_R((R:_Rt))$ for all $t\in T\backslash R$. That is, $\Rad_R((R:_Rt))$ is the same ideal for all $t\in T\backslash R$. They show in \cite[Lemma~2.11]{CahenDobbsLucas} that if an extension has a critical ideal, then the ideal is prime. Moreover, they show that if $R\subset T$ is a minimal extension, then the critical ideal exists \cite[Proposition~2.14(2)]{CahenDobbsLucas} and is maximal \cite[Theorem~3.5]{CahenDobbsLucas}. If $R\subset T$ has a critical ideal, we show $R^G\subset T^G$ has a critical ideal under any $G$-action such that $R^G\neq T^G$.

\begin{lemma}
\label{critical ideal invariance}
Let $P$ be the critical ideal of $R\subset T$. If $R^G\neq T^G$, then $\p:=P\cap R^G$ is the critical ideal of $R^G\subset T^G$.
\end{lemma}
\begin{proof}
Let $t\in T^G\backslash R^G$. Then $t\in T\backslash R$. Hence $P=\Rad_R((R:_Rt))$, from which it follows that
\[
\p=\Rad_R((R:_Rt))\cap R^G=\Rad_{R^G}((R:_Rt)\cap R^G)=\Rad_{R^G}((R^G:_{R^G}t)).
\]
Thus $\p$ is the critical ideal of $R^G\subset T^G$.
\end{proof}

We next show if a critical ideal is maximal, then its orbit (under $G$) is a singleton set.

\begin{lemma}
\label{critical ideal orbit}
Suppose $M=\Rad_R((R:_Rt))$, for all $t\in T\backslash R$, i.e., $M$ is the critical ideal for $R\subset T$. If $M$ is maximal, then $\sigma(M)=M$ for all $\sigma\in G$, i.e. $\mathcal{O}_M=\{M\}$.
\end{lemma}
\begin{proof}
Let $\sigma\in G$ and $t\in T\backslash R$. Note $\sigma^{-1}(t)\in T\backslash R$; otherwise, if $\sigma^{-1}(t)\in R$, then $t=\sigma(\sigma^{-1}(t))\in\sigma(R)=R$ -- contradiction. Since $M$ is the critical ideal for $R\subset T$, $M=\Rad_R((R:_R\sigma^{-1}(t)))$. Let $r$ be an arbitrary element of $R$, let $x\in M$, and set $y:=\sigma^{-1}(x)$. Then there exists $n\in\mathbb{N}$ such that $x^nr\in R$, from which it follows that $(\sigma^{-1}(x))^n\sigma^{-1}(t)\in \sigma^{-1}(R)=R$. Hence $y=\sigma^{-1}(x)\in\Rad_R((R:_R\sigma^{-1}(t)))=M$.  Thus $x=\sigma(y)\in\sigma(M)$, which shows $M\subseteq\sigma(M)$. Since $M$ is maximal, $M=\sigma(M)$, as desired.
\end{proof}

\begin{remark}
It is not necessary to assume $M$ is maximal in the preceding lemma. A similar set-theoretic argument establishes the converse $\sigma(M)\subseteq M$.
\end{remark}

Related to critical ideals are valuation pairs for an extension $R\subset T$. As in the introduction and \cite{Manis}, for $P\in\Spec(R)$, $(R,P)$ is a valuation pair of $T$ if there is a valuation $v$ on $T$ with $R=\{t\in T\suchthat v(t)\geq 0\}$ and $P=\{t\in T\suchthat v(t)>0\}$. Equivalently, $(R,P)$ is a valuation pair of $T$ if $R=S$ whenever $S$ is an intermediate ring containing a prime ideal lying over $P$ \cite{Manis}. Rank 1 valuation pairs are one of several equivalences of integrally closed minimal extensions given by Cahen et al \cite{CahenDobbsLucas}. As previously mentioned, the rank of a valuation pair $(R,P)$ of $T$ is the rank of the valuation group. The following lemma describes the relationship between critical ideals and valuation pairs.

\begin{lemma}\cite[Lemma~2.12]{CahenDobbsLucas}
\label{critical ideals and valuation pairs}
Let $(R,P)$ be a valuation pair of $T$. Then $R\subset T$ has a critical ideal if an only if $(R,P)$ has rank 1. Moreover, under these conditions, $P$ is the critical ideal of $R\subset T$.
\end{lemma}

Our next result is fundamental to the invariance of integrally closed minimal extensions established in Theorem~\ref{integrally closed mre invariance}.

\begin{proposition}
\label{valuation pair invariance}
Assume $G$ is locally finite such that $R^G\neq T^G$. Let $M\in\Max(R)$ and set $\m:=M\cap R^G$. If $\mathcal{O}_M=\{M\}$, then $(R^G,\m)$ is a valuation pair of $T^G$ whenever $(R,M)$ is a valuation pair of $T$.
\end{proposition}
\begin{proof}
Let $A$ be a ring such that $R^G\subseteq A\subseteq T^G$. Then $R\subseteq AR\subseteq T$. First note $AR$ is integral over $A$, since $R$ is integral over $R^G$, hence over $A$. Let $\q\in\Spec(A)$ such that $\q\cap R^G=\m$, and let $Q\in\Spec(AR)$ lie over $\q$. From
\[
\m=\q\cap R^G=(Q\cap A)\cap R^G=Q\cap R^G=(Q\cap R)\cap R^G
\]
it follows $Q\cap R$ is maximal in $R$, by integrality. We claim $Q\cap R=M$. Suppose not. Then there exists $x\in (Q\cap R)\backslash M$, since $Q\cap R$ and $M$ are incomparable (as maximal ideals). It follows $\tilde x\in Q\cap R^G=\m=M\cap R^G$. Hence $\sigma(x)\in M$ for some $\sigma\in G$. Since $\mathcal{O}_M=\{M\}$, we have $x\in\sigma^{-1}(M)=M$ -- contradiction. Hence $Q\cap R=M$. Since $(R,M)$ is a valuation pair of $T$, we have $AR=R$, whence $A=R^G$. Thus $(R^G,\m)$ is a valuation pair of $T^G$.
\end{proof}

Of the several integrally closed minimal extension equivalences in \cite[Theorem~3.5]{CahenDobbsLucas}, we use the condition that there exists a maximal ideal $M$ such that $(R,M)$ is a rank 1 valuation pair of $T$ where $R\subset T$. With this equivalence, it follows easily from the preceding results that integrally closed minimal  extensions are invariant under locally finite group action.

\begin{theorem}
\label{integrally closed mre invariance}
Assume $G$ is locally finite. If $R \subset T$ is an integrally closed minimal extension, then $R^G\subset T^G$ is an integrally closed minimal extension.
\end{theorem}
\begin{proof}
First we show $R^G\neq T^G$. Let $t\in T\backslash R$. Then $\tilde t\in T^G$. Suppose $\tilde t\in R^G$. Then $\tilde t\in R$. By \cite[Proposition~3.1]{FO}, $\sigma(t)\in R$ for some $\sigma\in G$, whence $t=\sigma^{-1}(\sigma(t))\in\sigma^{-1}(R)=R$ -- contradiction. Hence, $\tilde t\in T^G\backslash R^G$. Thus, $R^G\subsetneq T^G$.

Let $M$ be the critical ideal for $R\subset T$. By Lemma~\ref{critical ideal invariance}, $m:=M\cap R^G$ is the critical ideal for $R^G\subset T^G$. Since $R\subset T$ is a minimal extension, the critical ideal $M$ is maximal. By Lemma~\ref{critical ideal orbit} $\mathcal{O}_M=\{M\}$. By Lemma~\ref{valuation pair invariance} $(R^G,m)$ is a valuation pair of $T^G$. Since $m$ is the critical ideal of $R^G\subset T^G$, this valuation pair has rank 1 by Lemma~\ref{critical ideals and valuation pairs}. Hence, $R^G\subset T^G$ is an integrally closed minimal extension by \cite[Proposition~3.5]{CahenDobbsLucas}. 
\end{proof}

\section{Minimal Extensions, Flat Epimorphisms, and Normal Pairs}
\label{related results}

In this section, we generalize the results of Sections~\ref{integral section} and~\ref{integrally closed section}. Of course, arbitrary integral (integrally closed) extensions are a generalization of minimal integral (integrally closed) extensions. It is easy to see in Propositions~\ref{integrality invariance} and~\ref{integrally closed invariance} that integral and integrally closed extensions are invariant. 

In Proposition~\ref{mre invariance} and Corollary~\ref{mre invariance corollary}, we  show integral minimal extensions are invariant under stronger assumptions on $G$ and without the restriction of characteristic used in Theorem~\ref{integral mre invariance}. In doing so, we simultaneously re-establish Theorem~\ref{integrally closed mre invariance}.

In Theorem~\ref{perfect localization}, we exchange a stronger assumption for a more general result. In particular, we assume $G$ is strongly locally finite in order to show flat epimorphic extensions are invariant.

Lastly in Corollary~\ref{normal pair invariance}, we show normal pairs are invariant. As in \cite{Davis}, we say $(R,T)$ is a \textit{normal pair} if $S$ is integrally closed in $T$ whenever $R\subseteq S\subseteq T$. Clearly, if $R\subset T$ integrally closed minimal extension, then $(R,T)$ is a normal pair.

\begin{proposition}
\label{integrality invariance}
If $R\subset T$ is an integral extension and $G$ is locally finite, then $R^G\subseteq T^G$ is an integral extension.
\end{proposition}
\begin{proof}
It follows from Lemma~\ref{integrality} and by transitivity \cite[Theorem 40]{Kaplansky}.
\end{proof}

\begin{proposition}
\label{integrally closed invariance}
If $R$ is integrally closed in $T$, then $R^G$ is integrally closed in $T^G$.
\end{proposition}
\begin{proof}
Let $u\in T^G$ be integral over $R^G$. Then $u\in T$ is integral over $R$. Hence $u\in T^G\cap R=R^G$.
\end{proof}

As in Theorems~\ref{integral mre invariance} and~\ref{integrally closed mre invariance}, certain integral minimal extensions and all integrally closed minimal extensions are invariant under locally finite $G$-action. In the former, however, we require a certain restriction of characteristic. Assuming $|G|$ is finite and a unit in the base ring, we can remove this restriction. Of course, if $G$ is finite, then it is locally finite. Hence, the following result and corollary re-establish Theorem~\ref{integrally closed mre invariance}.

\begin{proposition}
\label{mre invariance}
Let $R\subset T$ be a minimal extension. Assume $G$ is finite such that $|G|$ is a unit in $R$ and $R^G\neq T^G$. Then $R^G\subset T^G$ is a minimal extension.
\end{proposition}
\begin{proof}
Let $u\in T^G\backslash R^G$. Clearly, $u\in T\backslash R$. Hence, $T=R[u]$. Let $t\in T^G$. Then $t=r_nu^n+\cdots+r_1u+r_0$ for some $r_i\in R$. Applying the averaging technique introduced in Section~\ref{integral section} we have 
\[
t=|G|^{-1}\sum_{\sigma\in G}\sigma(r_n)u^n+\cdots+\sigma(r_1)u+\sigma(r_0).
\]
Thus $T^G=R^G[u]$, i.e. $R^G\subset T^G$ is a minimal extension.
\end{proof}

Combining Propositions~\ref{integrality invariance},~\ref{integrally closed invariance}, and~\ref{mre invariance}, we have the following corollary.

\begin{corollary}
\label{mre invariance corollary}
Under the hypotheses of {\rm Proposition~\ref{mre invariance}}, if $R\subset T$ is an integral or integrally closed minimal extension, then $R^G\subset T^G$ is an integral or integrally closed minimal extension, respectively.
\end{corollary}

Integrally closed minimal extensions are flat epimorphic extensions (in the category of commutative rings), by \cite[Th\'eorm\`e~2.2]{FO}. Equivalently, flat epimorphisms are \textit{perfect localizations}, so-called because of the following correspondence.

\begin{theorem}\cite[Theorem~2.1,~Ch.~XI]{Stenstrom}
\label{flat epi}
Let $\phi:R\rightarrow T$ be a ring homomorphism. Then $\phi$ is a flat epimorphism if and only if the collection  $\mathcal{F}=\{I\subset R\,|\,\phi(I)T=T\}$ where $I$ is an ideal in $R$  is a Gabriel filter, and there exists an isomorphism $\psi:T\rightarrow R_{\mathcal{F}}$ such that $\psi\circ\phi:R\rightarrow R_{\mathcal{F}}$ is the canonical homomorphism. Such a filter is called perfect.
\end{theorem}

A collection of ideals $\mathcal{F}$ of a ring $R$ is a \textit{Gabriel filter} if it satisfies:
\begin{enumerate}[(i)]
\item If $I\in\mathcal{F}$ and $I\subseteq J$, then $J\in\mathcal{F}$.
\item If $I,J\in\mathcal{F}$, then $I\cap J\in\mathcal{F}$.
\item If for an ideal $I$ there exists $J\in\mathcal{F}$ such that $(I:j)\in\mathcal{F}$ for every $j\in J$, then $I\in\mathcal{F}$.
\end{enumerate}
For more information on Gabriel filters, see \cite{Stenstrom}.

By \cite[Exercise~8,~p.\ 242]{Stenstrom}, $T$ is a perfect localization of $R$ if and only if for all $t\in T$, $(R:_Rt)T=T$. With this definition and Lemma~\ref{filter lemma} we show perfect localizations (equivalently, flat epimorphic extensions) are invariant in Proposition~\ref{perfect localization}. 

\begin{lemma}
\label{filter lemma}
Assume $G$ is strongly locally finite. Define $\mathcal{F}:=\{I\subset R\,|\,IT=T\}$ and $\mathcal{F}':=\{J\subset R^G\,|\,JT^G=T^G\}$. If $I\in\mathcal{F}$, then $I\cap R^G\in\mathcal{F}'$.
\end{lemma}
\begin{proof}
Note $I\in\mathcal{F}$ if and only if every $P\in\Spec(R)$ containing $I$ is not lain over in $T$. Also note $\mathcal{F'}=\{J\subset R^G\,|\, JR\in\mathcal{F}\}$. Let $I\in\mathcal{F}$ and let $P\in\Spec(R)$ contain $(I\cap R^G)R$. We claim $I\subseteq\sigma(P)$ for some $\sigma\in G$, whence $PT=\sigma^{-1}(\sigma(P)T)=\sigma^{-1}(\sigma(PT))=T$ (since $IT=T$). Let $x\in I$. Then $\tilde x \in I\cap R^G$, so $\tilde x\in P$. It follows $\sigma(x)\in P$ for some $\sigma\in G$; equivalently, $x\in\sigma^{-1}(P)$. Hence $I\subseteq \bigcap_{Q\in\mathcal{O}_P}Q$. Since $G$ is strongly locally finite, $\mathcal{O}_P$ is finite. It follows that $I\subseteq Q$ for some $Q\in\mathcal{O}_P$ by the Prime Avoidance Lemma \cite[Theorem~81]{Kaplansky}. Hence the claim is satisfied by $\sigma\in G$, where $Q=\sigma(P)$, so $PT=T$. Thus, every prime containing $(I\cap R^G)R$ is not lain over in $T$. That is, $(I\cap R^G)R\in\mathcal{F}$, whence $I\cap R^G\in\mathcal{F'}$, as desired.
\end{proof}

We are now ready to show perfect localizations (flat epimorphic extensions) are invariant under strongly locally finite group action using Lemma~\ref{filter lemma}. 

\begin{theorem}
\label{perfect localization}
Let $G$ be strongly locally finite, and let $\mathcal{F}$ and $\mathcal{F'}$ be as in Lemma~\ref{filter lemma}. Then
\begin{inparaenum}[(a)] 
\item $\mathcal{F'}$ is a Gabriel filter whenever $\mathcal{F}$ is a Gabriel filter, and 
\item $T^G=(R^G)_{\mathcal{F'}}$ whenever $T=R_{\mathcal{F}}$.
\end{inparaenum}
In particular, if $R\subseteq T$ is a flat epimorphic extension, then so is $R^G\subseteq T^G$.
\end{theorem}
\begin{proof}
\begin{inparaenum}[(a)]
\item Suppose $\mathcal{F}$ is a Gabriel filter. We check that $\mathcal{F}'$ satisfies the defining conditions (i) through (iii) of a Gabriel filter given above. Let $I\in\mathcal{F}'$, and let $J$ be an ideal of $R^G$ containing $I$. Then $IR\in\mathcal{F}$ and $IR\subseteq JR$, so $JR\in\mathcal{F}$. It follows that $JT=T$, so $JT^G=T^G$, since $T$ is integral over $T^G$. Hence $J\in\mathcal{F}'$, which establishes condition (i). Now let $I,J\in\mathcal{F}'$. Then $IT=T$ and $JT=T$. Suppose $I\cap J\notin\mathcal{F}'$, i.e. $(I\cap J)T^G\neq T^G$. Again by integrality, $(I\cap J)T\neq T$. Let $P\in\Spec(T)$ contain $(I\cap J)T$. Then $I\cap J\subseteq P\cap T^G=:\p$. It follows that $I\subseteq \p$ or $J\subseteq \p$, but then $IT\subseteq P$ or $JT\subseteq P$ -- contradiction. Hence $I\cap J\in\mathcal{F}'$, which establishes condition (ii). 

It remains to show $\mathcal{F'}$ satisfies condition (iii). Let $J$ be an ideal of $R^G$, and suppose there exists $I\in\mathcal{F'}$ such that $(J:_{R^G}a)\in\mathcal{F'}$ for all $a\in I$. We claim $(JR:_Ra)\in\mathcal{F}$ for all $a\in IR$, whence $JR\in\mathcal{F}$, i.e., $J\in\mathcal{F'}$. Let $a:=a_1r_1+\cdots+a_nr_n\in IR$, where $a_i\in I$ and  $r_i\in R$. For each $a_i$, clearly $(J:_{R^G}a_i)R\subseteq(JR:_Ra_i)$. Since $(J:_{R^G}a_i)\in\mathcal{F}'$, we have $(J:_{R^G}a_i)R\in\mathcal{F}$. Hence $(JR:_Ra_i)\in\mathcal{F}$. From $(JR:_Ra_i)\subseteq(JR:_Ra_ir_i)$ it follows that $(JR:_Ra_ir_i)\in\mathcal{F}$. Since $\bigcap_{i=1}^n(JR:_Ra_ir_i)\in\mathcal{F}$ and $\bigcap_{i=1}^n(JR:_Ra_ir_i)\subseteq(JR:_Ra)$, we have $(JR:_Ra)\in\mathcal{F}$, proving the claim. Hence $JR\in\mathcal{F}$, i.e. $J\in\mathcal{F}'$. Thus $\mathcal{F'}$ is a Gabriel filter.

\item Now we show $T^G=(R^G)_{\mathcal{F}'}$ by showing that $T^G$ is a perfect localization of $R^G$. Let $x\in T^G$. Then $(R:_Rx)T=T$, since $T$ is a perfect localization of $R$. It follows that $(R:_Rx)\in\mathcal{F}$, and $(R:_Rx)\cap R^G\in\mathcal{F'}$, by Lemma~\ref{filter lemma}. We claim $(R:_Rx)\cap R^G\subseteq (R^G:_{R^G}x)$, whence $(R^G:_{R^G}x)\in\mathcal{F'}$, since $\mathcal{F'}$ is a Gabriel filter. Let $y\in(R:_Rx)\cap R^G$. Then $xy\in R$, but $x\in T^G$ and $y\in T^G$, so $xy\in R^G$. Hence $(R:_Rx)\cap R^G\subseteq (R^G:_{R^G}x)$, so $(R^G:_{R^G}x)\in\mathcal{F}'$ as claimed. (In fact, as the reverse containment clearly holds, $(R:_Rx)\cap R^G= (R^G:_{R^G}x)$.) Thus $(R^G:_{R^G}x)T^G=T^G$, i.e. $T^G$ is a perfect localization of $R^G$. Equivalently, $R^G\subseteq T^G$ is a flat epimorphic extension whenever $R\subseteq T$ is a flat epimorphic extension.
\end{inparaenum}
\end{proof}

\begin{remark}
It would be interesting to know if epimorphic extensions or flat extensions are invariant under any group action.
\end{remark}

Normal pairs are another generalization of integrally closed minimal extensions. By \cite[Theorem~5.2]{KnebuschZhang}, $(R,T)$ is a normal pair if and only if $R$ is integrally closed in $T$ and $R\subseteq S$ satisfies INC for any intermediate ring $S$. We call a pair of rings $(R,T)$ satisfying the latter property an \textit{INC-pair} and note it is equivalent to the definition of an INC-pair given in \cite{DobbsLO}.

We have already seen integrally closed extensions are invariant in Proposition~\ref{integrally closed invariance}. To assert normal pairs are invariant, it remains to show INC-pairs are invariant. 

\begin{proposition}
\label{INC invariance}
Assume $G$ is locally finite. If $(R,T)$ is an INC-pair, then $(R^G,T^G)$ is an INC-pair.
\end{proposition}
\begin{proof}
Let $R^G\subseteq A\subseteq T^G$, and let $\q\subseteq \q'$ be prime ideals of $A$ with the same contraction in $R^G$. Set $\p:=\q\cap R^G=\q'\cap R^G$. Since $R$ is integral over $R^G$ (whence over $A$), $AR$ is integral over $A$. Hence, $A\subseteq AR$ satisfies LO and GU. Let $Q\subseteq Q'$ be prime ideals in $AR$ such that $\q=Q\cap A$ and $\q'=Q\cap A$. Setting $P:=Q\cap R$ and $P':=Q'\cap R$, we have $P\subseteq P'$ and 
\[
P\cap R^G=Q\cap R^G=(Q\cap A)\cap R^G=\q\cap R^G=\p,
\]
and $P'\cap R^G=\p$, by the same reasoning. As an integral extension, $R^G\subseteq R$ satisfies INC, whence $P=P'$. Since $R\subseteq AR$ satisfies INC,  $Q=Q'$. Hence $q=q'$. Thus $(R^G,T^G)$ is an INC-pair.
\end{proof}

The corollary below now follows easily from Propositions \ref{integrally closed invariance} and \ref{INC invariance}.

\begin{corollary}
\label{normal pair invariance}
If $G$ is locally finite, then $(R^G,T^G)$ is a normal pair whenever $(R,T)$ is a normal pair.
\end{corollary}

\begin{remark}
By \cite[Corollary~2.4(bis\.)]{DobbsLO}), \textit{P-extensions} are precisely INC-pairs. Hence P-extensions are invariant, by Proposition~\ref{INC invariance}.
\end{remark}

\section*{Acknowledgment}
The author is immensely grateful to her advisor, Jay Shapiro, for his guidance, suggestions, and help revising the manuscript.

\bibliographystyle{abbrv}
\bibliography{bib1}

\end{document}